\documentclass[review]{elsarticle}
\usepackage{amssymb}
\usepackage{amssymb}
\usepackage{mathrsfs}
\usepackage{stmaryrd}
\usepackage{amsfonts}
\usepackage{amsmath}    
\usepackage{amssymb}    
\usepackage[amsmath, thmmarks]{ntheorem} 
\usepackage{subfigure}                                 
\usepackage[top=0.8cm, bottom=0.8cm, left=1.9cm, right=1.9cm, dvips, dvipdfm, includehead, includefoot]{geometry}
\usepackage{indentfirst}               %
\usepackage{fancyhdr,  mathrsfs}
\usepackage{graphicx}
\usepackage{float}
\journal{Journal of Computational Physics}

\begin{document}
\newtheorem{algorithm}{\bf{Algorithm}}[section]
\newtheorem{assumption}{\bf{Assumption}}[section]
\newtheorem{lemma}{\bf{Lemma}}[section]
\newtheorem{theorem}{\bf{Theorem}}[section]
\newtheorem{definition}{\bf{Definition}}[section]
\newtheorem{proposition}{\bf{Proposition}}[section]
\newtheorem{remark}{\bf{Remark}}[section]
\newtheorem{example}{\bf{Example}}[section]
\newtheorem{proof}{\it{Proof.}}
\begin{frontmatter}

\title{A positivity-preserving finite volume element method for anisotropic diffusion problems on quadrilateral meshes}

\tnotetext[mytitlenote]{This work was partially supported by the Postdoctoral Science Foundation of China (no.2017M620689), the National Science Foundation of China (nos. 11571048 and 11401034), and the CAEP Developing Fund of Science and Technology (no. 2014A0202009).}

\author[mymainaddress]{Yanni Gao}
\ead{gaoyn10@163.com}

\author[mymainaddress,mysecondaryaddress]{Guangwei Yuan}
\ead{yuan\_guangwei@iapcm.ac.cn}

\author[mymainaddress]{Shuai Wang}
\ead{wang\_shuai@iapcm.ac.cn}

\author[mymainaddress,mysecondaryaddress]{Xudeng Hang\corref{mycorrespondingauthor}}
\cortext[mycorrespondingauthor]{Corresponding author}
\ead{hang\_xudeng@iapcm.ac.cn}

\address[mymainaddress]{Institute of Applied Physics and Computational Mathematics, Fenghaodong Road, Haidian District, Beijing 100094, China}
\address[mysecondaryaddress]{Laboratory of Computational Physics, PO Box 8009, Beijing, 100088, China}

\begin{abstract}
In this paper, we propose a nonlinear positivity-preserving finite volume element(FVE) scheme for anisotropic diffusion problems on quadrilateral meshes. Based on an overlapping dual partition, the one-sided flux is approximated by the iso-parametric bilinear element. A positivity-preserving nonlinear scheme with vertex-centered unknowns is obtained by a new two-point flux technique, which avoids the convex decomposition of co-normals and the introduction of intermediate unknowns. The existence of a solution is proved for this nonlinear system by applying the Brouwer's theorem. Numerical results show that the proposed positivity-preserving scheme is effective on distorted quadrilateral meshes and has approximate second-order accuracy for both isotropic and anisotropic diffusion problems. Moreover, the presented scheme is applied on an equilibrium radiation diffusion problem with discontinuous coefficients. Numerical results show that the new scheme is much more efficient than the standard FVE method.
\end{abstract}

\begin{keyword}
Finite volume element method, Positivity-preserving, Anisotropic diffusion, Quadrilateral mesh, Equilibrium radiation diffusion
\end{keyword}

\end{frontmatter}


\section{Introduction}
\label{intro}
\noindent Positivity-preserving or ¡°monotonicity¡± is a significant requirement for discrete schemes of diffusion equations. A scheme without such a property can lead to non-physical solutions or oscillation. In some applications, such as Lagrangian computation of hydrodynamic problems and reservoir simulations, meshes are usually distorted and the coefficients are heterogeneous and anisotropic. These two factors may cause diffusion schemes more readily to produce non-physical negative solutions. To avoid this problem, considerable attention has been given to the development of numerical schemes which respect the positivity-preserving property.

Owing to local conservation of finite volume(FV) methods, several positivity-preserving FV schemes have been proposed\cite{W3,W12,W14,W16,WW4,W26,W35}. It is shown in \cite{J12} that no linear and monotone nine-point scheme with second-order accuracy exists on any distorted quadrilateral mesh or for any anisotropic diffusion. Therefore, to get a second-order monotone scheme, one must reference nonlinear monotone schemes. In \cite{J13}, a nonlinear FV scheme is proposed for a highly anisotropic diffusion operator on unstructured triangular meshes. Since its proposal, this approach has been further developed to obtain schemes that preserve positivity\cite{W3,W14,W16,W26,W35} and maximum value principle\cite{W12,WW4} on general grids. In \cite{W35}, a nonlinear FV scheme with little severe restriction to the collocation points is presented, which has been further improved in \cite{W26,S1} and extended to non-equilibrium radiation diffusion problems\cite{J19} and advection diffusion problems\cite{J21}. In these works, the convex decomposition of co-normals and the positivity-preserving interpolation of auxiliary unknowns are two key components. However, the convex decomposition of co-normals can result in a non-fixed stencil. Moreover, the positivity-preserving interpolation of auxiliary unknowns is hard to guarantee on general distorted meshes or for anisotropic diffusion problems if we strive to obtain a scheme with accuracy higher than the first order.

In recent years, improved FV schemes have been further developed. In \cite{S1}, a nonlinear FV scheme with second-order accuracy is proposed, which avoids the assumption that values of auxiliary unknowns are nonnegative. In \cite{W17}, the authors present an interpolation-free approach based on local repositioning of cell centers, though it is applicable for cells with only one discontinuous face. In \cite{W6}, a nonlinear scheme with two sets of primary unknowns is constructed, which is interpolation-free. Nevertheless, this scheme can not achieve the second-order accuracy in the case of discontinuous diffusion coefficients. Nonlinear two-point flux approximations that differ from the above approaches are further developed in \cite{CV,W13,Wu}, where neither the positive interpolation of auxiliary unknowns nor the convex decomposition of co-normals is required.

The finite volume element(FVE) method\cite{Ref16}, also called the generalized difference method\cite{Refff16}, has the properties of simple calculation and local conservation. Compared with FV methods, an advantage of FVE methods is the theoretical basis. Recently, many studies have been devoted to developing higher-order FVE schemes and establishing their error estimations. The readers can refer to the papers\cite{Ref26,QiHongJimingWu,Ref22,RRR8} and the references therein. However, almost all of their construction and corresponding theoretical analysis are executed for problems with scalar or even constant diffusion coefficients, and they focus on the proofs of coercivity and optimal convergence rates. The development of FVE schemes satisfying positivity preservation is far more difficult to conduct, especially for problems with anisotropic diffusion. This difficulty has impeded the application of FVE methods on some complicated problems. Although there exist some works devoted to the application of FVE methods on complicated numerical simulations\cite{ShuShiHuangYunqing,L2}, these schemes are not monotone and the non-physical solutions are just modified by a simple cutoff method\cite{cutoff} or repair techniques\cite{M14}. To our best knowledge, few monotone FVE schemes have been found until now. Therefore, the primary motivation of the present study is to construct a nonlinear monotone FVE scheme with approximate second-order accuracy, which can be applied to real practical problems.

In this paper, we construct a nonlinear monotone FVE scheme in the framework of the iso-parametric bilinear FVE method on quadrilateral meshes. The traditional barycenter dual partition is replaced by an overlapping one, which make it possible to construct a positivity-preserving FVE scheme. Firstly, we employ the iso-parametric bilinear element to approximate the gradients on the barycenters, which is used to get a continuous one-sided flux on each edge of dual elements. Finally, a new two-point flux technique is applied to obtain a monotone scheme with vertex-centered unknowns. In order to make the resulted scheme regular with respect to the unknowns, two parameters are introduced which can still keep the scheme's second-order accuracy. Accordingly, we prove the existence of a solution for this nonlinear system by the Brouwer's theorem. It should be pointed out that the nonlinear two-point flux technique employed in this paper differs from those considered in \cite{CV,W13,Wu} and is also applicable to other second-order schemes. Compared with the standard FVE scheme\cite{ShuShiHuangYunqing,Shu16}, the proposed monotone FVE scheme has comparable accuracy and can significantly reduce the computational costs. Besides, the construction and implementation of the new scheme are simpler than FV schemes, because the auxiliary unknowns interpolation and the decomposition of co-normals are not needed.

The remainder of this paper is organized as follows. We briefly introduce the standard FVE method on quadrilateral meshes in Section 2. In Section 3, we describe the construction of the monotone FVE scheme and prove the existence of a solution for this new scheme. In Section 4, numerical results are presented to evaluate the performance of the proposed FVE scheme. Finally, conclusions are provided in Section 5.
\section{Preliminaries of the standard FVE method}
\label{sec:2}
Consider the unsteady diffusion problem:
\begin{align}
\label{eq:F1}
\partial_{t}u-\nabla \cdot(\kappa \nabla u ) &= f  & \text{in~}\, \Omega\times(0,T),\\
\label{eq:F2}
\gamma (\kappa \nabla u)\cdot \mathbf{n}+\delta u&= g & \text{on~}\, \partial \Omega\times(0,T),
\end{align}
where

\begin{description}
  \item[(a)] $u=u(x,t)$ is the solution of the diffusion equation;
  \item[(b)] $\Omega$ is an open-bounded, convex-connected polygonal domain in $\mathbb{R}^2$ with boundary $\partial \Omega$, and $T$ is a positive constant;
  \item[(c)] The source term $f\in L^2(\Omega\times(0,T))$ and the boundary data $g\in H^{1/2}(\partial\Omega\times(0,T))$;
  \item[(d)] $\kappa$ is a symmetric tensor such that $\kappa$ is piecewise continuous on $\Omega\times(0,T)$. It may be discontinuous on some interfaces, and the set of eigenvalues of $\kappa$ is included in $[\lambda_{\text{min}},\lambda_{\text{max}}]$ with $\lambda_{\text{min}}>0$.
  \item[(e)] $\gamma$ and $\delta$ are smooth functions such that
  \begin{equation*}
    \forall~ (x,t)\in \partial\Omega\times(0,T),~~~   \delta(x,t)\geq0, ~\text{and}~ \gamma(x,t)\geq \gamma_0>0.
  \end{equation*}
\end{description}
The initial data is given by
\begin{equation*}
u(x,0)=u_0(x)\in H^1(\Omega).
\end{equation*}

Divide $\overline{\Omega}$ into a set of strictly convex quadrilaterals such that different quadrilaterals have no common interior point. Accordingly, the vertexes of any quadrilateral do not lie on the interior of an edge of any other quadrilateral, and any vertex on the boundary is a vertex of some quadrilateral. Thus, we obtain a primary partition $\mathcal{T}_h$, where $h$ is the largest diameter of all quadrilaterals. For any element
$K\in \mathcal{T}_h$ with vertexes $P_i,i=1,2,3,4,$ as shown in Figure \ref{Fig1}, let $M_i$ be the midpoint of
edge $P_iP_{i+1}(P_{5}=P_{1})$ and $Q$ be the barycenter of $K$. Let $\Omega_{h}$ be the set of nodes of the mesh. The number of all nodes is denoted by $\mathcal{N}$.
\begin{figure}
\centerline
{
\includegraphics[width=2.2in]{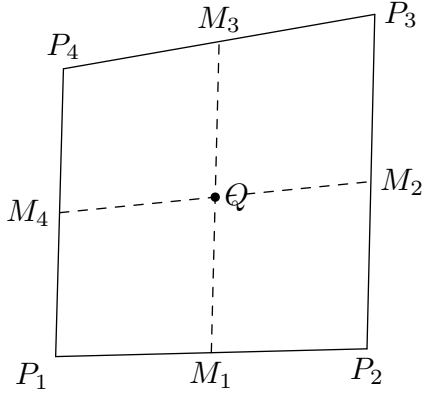}
}
\centering
\caption{Traditional control volumes in a quadrilateral element}

\label{Fig1}
\end{figure}

Let $h_K$ denote the diameter of the element $K\in \mathcal{T}_h$, $\rho_K$ denote the maximum
diameter of circles contained in $K$, and $S_K$ denote the area of K. Suppose $\mathcal{T}_h$ is regular, i.e., there exist two positive constants $C$ and $C_1$ such that for any element $K\in \mathcal{T}_h$
\begin{eqnarray}
&&\frac{h_K}{\rho_K}\leq C, \label{AS1}\\
&&\mid \cos \theta_K \mid\leq C_1<1, \text{where}~\theta_K ~\text{is any interior angle of}~K. \label{AS2}
\end{eqnarray}

To define the trial function space $U_h$, we take the unit square $\widehat{K}=[0,1]\times[0,1]$ on the $(\xi, \eta)$ plane as a reference element. For any convex quadrilateral $K$, there exists a unique invertible bilinear mapping $F_K$, which maps $\widehat{K}$ onto $K$.
\begin{equation}\label{EF3}
F_K=\left\{\begin{array}{ll}
x=x_1+a_1\xi+a_2\eta+a_3\xi\eta, \\
y=y_1+b_1\xi+b_2\eta+b_3\xi\eta,
\end{array}
\right.
\end{equation}
where
\begin{eqnarray*}
a_1=x_2-x_1, a_2=x_4-x_1, a_3=x_1-x_2+x_3-x_4, \\
b_1=y_2-y_1,~ b_2=y_4-y_1,~ b_3=y_1-y_2+y_3-y_4,
\end{eqnarray*}
and $(x_i,y_i), i=1,2,3,4,$ are the coordinates of vertices $P_i(i=1,2,3,4)$ on element $K$.
Denote the Jacobi matrix of the mapping $F_K$ by $J_K(\xi,\eta)$, then
\begin{equation}\label{eqF4}
J_K(\xi,\eta)=\left(
                          \begin{array}{cc}
                            \frac{\partial x}{\partial \xi} & \frac{\partial x}{\partial \eta} \\
                            \frac{\partial y}{\partial \xi} & \frac{\partial y}{\partial \eta} \\
                          \end{array}
                        \right)=\left(
                                  \begin{array}{cc}
                                    a_1+a_3\eta & a_2+a_3\xi \\
                                    b_1+b_3\eta & b_2+b_3\xi \\
                                  \end{array}
                                \right).
\end{equation}
Let $\mathcal{J}_{Q}$ be the determinant of the Jacobi matrix with respect to $Q$.
Simple calculation shows that
\begin{eqnarray*}
 &&\mathcal{J}_{Q}=S_K.
\end{eqnarray*}
Due to the regularity of $\mathcal{T}_h$, there exists a positive constant $C$ such that
\begin{eqnarray}\label{reg1}
\mathcal{J}_{Q}\geq Ch_{K}^{2}.
\end{eqnarray}

Denote by $U_h(\widehat{K})$ the iso-parametric bilinear polynomial space on $\widehat{K}$, then the trial function space $U_h$ is defined as
\begin{eqnarray}\label{eqf5}
U_h=\{u_h\in H^{1}(\overline{\Omega}): u_h\mid_K=\widehat{u}_h\circ F_{K}^{-1}, \widehat{u}_h \in U_h(\widehat{K}), K\in \mathcal{T}_h\},
\end{eqnarray}
where $\widehat{u}_h\circ F_{K}^{-1}$ denotes the compound function of $\widehat{u}_h$ and the inverse mapping $F_{K}^{-1}.$

The control volume of vertex $P_i$ in $K$, as shown in Figure \ref{Fig1}, is the subregion $M_{i-1}P_iM_iQ, i=1,2,3,4$, where $M_0=M_4.$
Hence, the control volume associated with vertex $P$ is denoted by $K_{P}^{\ast}$, which is the union of the above subregions sharing the vertex $P$. The set of all control volumes forms a dual partition $\mathcal{T}_{h}^{\ast}$.

Define the test function space as
\begin{equation}\label{eqf6}
V_h=\{v_h\in L^2(\Omega): v_h\mid_{K_{P}^{\ast}}=\text{constant}, \forall ~P\in \Omega_{h}\}.
\end{equation}

The standard FVE scheme for (\ref{eq:F1})-(\ref{eq:F2}) on quadrilateral meshes is to find $u_h(x,\cdot)\in U_h$, such that
\begin{equation}\label{eqF7}
(\partial_{t}u_h,v_h)+a_h(u_h,v_h)=(f,v_h)~\forall~v_h\in V_h
\end{equation}
with
\begin{eqnarray}
a_h(u_h,v_h)=-\sum_{K_{P}^{\ast}\in \mathcal{T}_{h}^{\ast}} \int_{\partial K_{P}^{\ast}} \kappa \nabla u_h\cdot \mathbf{n} v_h ds=-\sum_{K_{P}^{\ast}\in \mathcal{T}_{h}^{\ast}}v_h(P)\int_{\partial K_{P}^{\ast}} \kappa \nabla u_h\cdot \mathbf{n}ds, \label{eqF8}
\end{eqnarray}
\begin{eqnarray}
(f,v_h)=\sum_{K_{P}^{\ast}\in \mathcal{T}_{h}^{\ast}}\int_{K_{P}^{\ast}} fv_h dx=\sum_{K_{P}^{\ast}\in \mathcal{T}_{h}^{\ast}}v_h(P)\int_{K_{P}^{\ast}} fdx \label{eqF9},
\end{eqnarray}
where $P\in \Omega_{h}$ and $\mathbf{n}$ is the outward unit normal vector of $\partial K_{P}^{\ast}$.

In \cite{Refff16}, detailed proofs are provided for the coercivity and error estimations of the FVE scheme (\ref{eqF7}). However, a study on the monotonicity of FVE methods is difficult to conduct and few works can be found. Owing to the nonlinearity of the iso-parametric bilinear element gradient, constraints on computational meshes and material properties prevent FVE methods from practical applications. The numerical results presented in \cite{L2} have shown that the standard FVE schemes can lead to negative solutions which are non-physical and make the computation terminate. To overcome this drawback, authors in \cite{L2} employ two nonnegative modifications, the cutoff method in \cite{cutoff} and repair techniques in \cite{M14}, to handle this problem. However, the cutoff method suffers from the conservation error, and the local repair technique increases the computational cost and breaks the governing equations. Therefore, it is urgent to construct a positivity-preserving FVE scheme.
\section{Monotone FVE scheme}
\label{sec:4}
In this section, we construct a monotone FVE scheme for the spacial discretization. For simplicity and without confusion, we omit the variate $t$ in the following description.

To overcame the difficulties of standard FVE method, we define a new dual partition formed by the diagonals of all quadrilateral elements. For a node $P$, as shown in Figure \ref{Fig2}, let the shaded area be the control volume associated with $P$, which is denoted by $K_{P}^{\ast}.$ The set of all such control volumes forms a new dual partition, which is also denoted by $\mathcal{T}_{h}^{\ast}$. Obviously, this new dual partition is overlapping.
\begin{figure}[h]
\centerline
{
\includegraphics[width=2.2in]{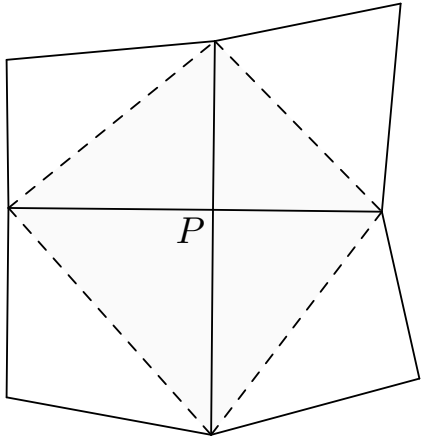}
}
\centering
\caption{New control volume for a node $P$.}
\label{Fig2}
\end{figure}
We still denote by $V_h$ the test function space associated with this new dual partition.
The definition of $V_h$ is similar to (\ref{eqf6}), except its support set.

Applying $V_h$ to test problem (\ref{eq:F1})-(\ref{eq:F2}), we obtain the corresponding weak formulation by Green's theorem:
\begin{equation}\label{eqF10}
(\partial_{t}u,v_h)+\sum_{K_{P}^{\ast}\in \mathcal{T}_{h}^{\ast}}v_h(P)\sum_{\sigma\in \partial K_{P}^{\ast} }\mathcal{F}_{P,\sigma}=\sum_{K_{P}^{\ast}\in \mathcal{T}_{h}^{\ast}}v_h(P)\int_{K_{P}^{\ast}} fdx~~ ~\forall~v_h\in V_h
\end{equation}
with
\begin{eqnarray}
\mathcal{F}_{P,\sigma}=-\int_{\sigma} \kappa \nabla u\cdot \mathbf{n}ds=-\int_{\sigma} \nabla u\cdot \kappa^{T}\mathbf{n}ds \label{eqF11}
\end{eqnarray}
where $P\in \Omega_{h}$ and $\mathbf{n}$ is the outward unit normal vector of $\sigma$.

We present a discrete flux $F_{P,\sigma}$ to approximate $\mathcal{F}_{P,\sigma}$. Unlike any FV method, we do not need to decompose the co-normal $\kappa^{T}\mathbf{n}$. We directly discretize $\nabla u$ by the iso-parametric bilinear element.

We take nodes $P_1, P_3$ and an edge $\sigma=P_2P_4$ (as shown in Figure \ref{Fig3}) to present the discretization of continuous fluxes $\mathcal{F}_{P_1,\sigma}, \mathcal{F}_{P_3,\sigma}$. In Figure \ref{Fig3}, $\mathbf{q}_{i}, i=1,2,$ denote the normal vectors with respect to the diagonals. Assume that these vectors have the same lengths as their corresponding diagonals. For a given $u_h\in U_h$, let $u_i$ be the value of $u_h $ at node $P_i, i=1,2,3,4.$
\begin{figure}[h]
\centerline
{
\includegraphics[width=2.3in]{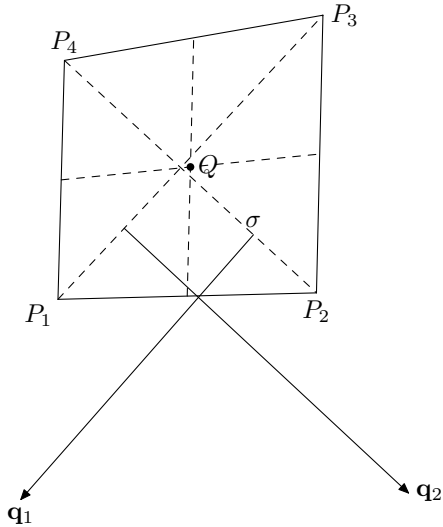}
}
\centering
\caption{Notations on a quadrilateral element.}
\label{Fig3}
\end{figure}

On an element as shown in Figure \ref{Fig3}, the trial function $u_h$ takes the following form
\begin{align}
u_h(x,y)=\widehat{u}_h(\xi,\eta)=u_1N_1(\xi,\eta)+u_2N_2(\xi,\eta)+u_3N_3(\xi,\eta) +u_4N_4(\xi,\eta)\nonumber
\end{align}
where
\begin{align}
&N_1(\xi,\eta)=(1-
\xi)(1-\eta),~~~~~ N_2(\xi,\eta)=\xi(1-\eta) \nonumber \\
&N_3(\xi,\eta)=\xi\eta,~~~~~~~~~~~~~~~~~~~~ N_4(\xi,\eta)=(1-\xi)\eta. \nonumber
\end{align}
By simple calculation, the gradient of $u_h$ at $Q$ is formulated as follows
\begin{align}\label{eqf14}
\nabla u_h\mid_{Q}=\frac{u_1-u_3}{2\mathcal{J}_{Q}}\mathbf{q}_{1}+\frac{u_2-u_4}{2\mathcal{J}_{Q}}\mathbf{q}_{2},
\end{align}
where $\mathcal{J}_{Q}$ is the determinant of the Jacobi matrix corresponding to $Q$.
\subsection{Construction of the monotone FVE scheme}
In this subsection, we describe the construction of the new monotone FVE scheme, which is a nonlinear two-point flux approximation.

Substitute $\nabla u_h\mid_{Q}$ given in (\ref{eqf14}) into (\ref{eqF11}), we obtain a one-sided flux $F_{P_1}$ with respect to node $P_1$
\begin{align}
F_{P_1}&=\nabla u_h\mid_{Q}\cdot \kappa_{Q}^{T}\mathbf{q}_{1}=A_{\sigma}(u_1-u_3)+r_{\sigma}(u_2,u_4),\label{eqF24}
\end{align}
where
\begin{align}
&A_{\sigma}=\frac{\mathbf{q}_{1}\cdot \kappa_{Q}^{T}\mathbf{q}_{1}}{2\mathcal{J}_{Q}}~~ \text{is the main part of the flux,}\nonumber\\
&r_{\sigma}(u_2,u_4)=\frac{\mathbf{q}_{2}\cdot \kappa_{Q}^{T}\mathbf{q}_{1}}{2\mathcal{J}_{Q}}(u_2-u_4)~~ \text{is the residual of the flux,} \nonumber
\end{align}
and $\kappa_{Q}^{T}=\kappa(Q)^{T}$. Firstly, we have the following estimation for $A_{\sigma}$
\begin{equation*}
0< A_{\sigma}\leq\frac{\lambda_{\text{max}}h_{K}^{2}}{Ch_{K}^{2}}=\frac{\lambda_{\text{max}}}{C},
\end{equation*}
where the positive definiteness of $\kappa_{Q}$ and the regularity (\ref{reg1}) have been used. Hence, the coefficient $A_{\sigma}$ is bounded. With respect to the node $P_3$, continuing to use $\nabla u_h\mid_{Q}$ to approximate $\nabla u$, we obtain a one-sided flux $F_{P_3}$
\begin{align}
F_{P_3}&=\nabla u_h\mid_{Q}\cdot \kappa_{Q}^{T}(-\mathbf{q}_{1})=A_{\sigma}(u_3-u_1)-r_{\sigma}(u_2,u_4).\label{eqF25}
\end{align}
Let
\begin{equation*}
r_{\sigma}^{+}=\frac{\mid r_{\sigma}(u_2,u_4)\mid+r_{\sigma}(u_2,u_4)}{2}, r_{\sigma}^{-}=\frac{\mid r_{\sigma}(u_2,u_4)\mid-r_{\sigma}(u_2,u_4)}{2}.
\end{equation*}
be the positive and negative parts of $r_{\sigma}(u_2,u_4)$, then the one-sided fluxes $F_{P_1}$ and $F_{P_3}$ can be further rewritten as
\begin{align}
F_{P_1}=A_{\sigma}(u_1-u_3)+r_{\sigma}^{+}-r_{\sigma}^{-}, \label{eqF26}\\
F_{P_3}=A_{\sigma}(u_3-u_1)+r_{\sigma}^{-}-r_{\sigma}^{+}. \label{eqFF26}
\end{align}

Assume that $u_h$ is positive at all vertexes and let $M, C$ be two positive constants independent of $u_h$. The value of $M$ will be discussed in next subsection and the value of $C$ is less than or equal to 1 in our numerical tests.

Based on $F_{P_1}, F_{P_3}$ given by (\ref{eqF26}) and (\ref{eqFF26}), we obtain the two-point flux approximations $F_{P_1,\sigma}, F_{P_3,\sigma}$ by the following nonlinear technique:
\begin{itemize}
  \item For the node $P_1$, we have
\begin{align}
F_{P_1,\sigma}=(A_{\sigma}+\frac{r_{\sigma}^{+}M}{Mu_1+Ch^2})u_1-(A_{\sigma}+\frac{r_{\sigma}^{-}M}{Mu_3+Ch^2})u_3. \label{eqF27}
\end{align}
  \item For the node $P_3$, we have
\begin{align}
F_{P_3,\sigma}=(A_{\sigma}+\frac{r_{\sigma}^{-}M}{Mu_3+Ch^2})u_3-(A_{\sigma}+\frac{r_{\sigma}^{+}M}{Mu_1+Ch^2})u_1. \label{eqF28}
\end{align}
\end{itemize}
Obviously, $F_{P_1,\sigma}, F_{P_3,\sigma}$ still satisfy the local conservation because $F_{P_1,\sigma}+F_{P_3,\sigma}=0$.

Specially, if $P_1$ is a boundary node and $\sigma\in \partial K_{P_1}^{\ast}$ is a boundary edge, then it holds that from the boundary condition (\ref{eq:F2})
\begin{align}
\mathcal{F}_{P_1,\sigma}=\int_{\sigma}(\frac{\delta}{\gamma}u-g)ds. \nonumber
\end{align}
We employ the following discretization $F_{P_1,\sigma}$ to approximate $\mathcal{F}_{P_1,\sigma}$
\begin{align}
F_{P_1,\sigma}=(\frac{\delta_{\sigma}}{\gamma_{\sigma}}u_{P_1}-g_{\sigma}) \mid \sigma \mid,  \nonumber
\end{align}
where $u_{P_1}$ is a unknown, $\mid \sigma \mid$ is the length of edge $\sigma$, and $\delta_{\sigma}, \gamma_{\sigma}$ and $g_{\sigma}$ are the approximations of $\delta, \gamma$ and $g$ on the edge $\sigma$ respectively. There are several ways to calculate $\delta_{\sigma}, \gamma_{\sigma}$ and $g_{\sigma}$, such as the midpoint value or area mean on the edge $\sigma$.

Finally, the monotone FVE scheme can be formulated as: Find $u_h(x,\cdot)\in U_h$, such that
\begin{equation}\label{eqF29}
(\partial_{t}u_h,v_h)+\sum_{K_{P}^{\ast}\in \mathcal{T}_{h}^{\ast}}v_h(P)\sum_{\sigma\in \partial K_{P}^{\ast} }F_{P,\sigma}=\sum_{K_{P}^{\ast}\in \mathcal{T}_{h}^{\ast}}v_h(P)\int_{K_{P}^{\ast}} fdx,~~~\forall~v_h\in V_h.
\end{equation}
Obviously, this monotone FVE scheme is nonlinear and there are five nonzero elements in each row. If the diffusion coefficient $\kappa$ is linear, the solution of this scheme is less time-consumming than linear schemes. Nevertheless, numerical tests show that, if the diffusion coefficient $\kappa$ is nonlinear, the monotone FVE scheme (\ref{eqF29}) needs fewer linear iterations per time step and reduce the costs by almost half compared with standard FVE method.

\begin{remark}
Although the monotone FVE scheme (\ref{eqF29}) is constructed for unsteady diffusion equations, it is likewise applicable for steady diffusion equations. For convenience, we take several steady diffusion problems to examine the monotonicity and convergence of this new scheme in our numerical tests.
\end{remark}
\begin{remark}
The monotone FVE scheme (\ref{eqF29}) is valid for the case that the diffusion coefficients are discontinuous, provided that the cell sides fit the interfaces.
\end{remark}

\subsection{Analysis of the truncation error}
In the nonlinear technique (\ref{eqF27})-(\ref{eqF28}), the two parameters $M$ and $C$ are introduced to ensure the scheme's regularity with respect to the unknowns. However, we will show that they do not break the second-order accuracy in numerical experiments. Next, we give the analysis of the truncation error for this new monotone FVE scheme.

Under the assumption that $u_h$ is positive at all vertexes, there exists a positive constant $M$ which is large enough such that
\begin{align}
Mu_h(P)\geq O(h),~\forall~P\in \Omega_h. \label{Adm}
\end{align}
Similar to the standard FVE method, under the assumptions (\ref{AS1})-(\ref{AS2}), the one-sided flux $(\ref{eqF24})$ satisfies
\begin{align}
\mid\mathcal{F}_{P_1,\sigma}-F_{P_1}\mid=O(h^2), \nonumber
\end{align}
which is a classical estimation of the finite element space $U_h$ \cite{Refff16}.
For the nonlinear two-point flux approximation $F_{P_1,\sigma}$ given by $(\ref{eqF27})$, we have
\begin{align}
\mid\mathcal{F}_{P_1,\sigma}-F_{P_1,\sigma}\mid&=\mid\mathcal{F}_{P_1,\sigma}-F_{P_1}+F_{P_1}-F_{P_1,\sigma}\mid \nonumber \\
&\leq O(h^2)+\frac{Cr_{\sigma}^{+}h^2}{Mu_1+Ch^2}+\frac{Cr_{\sigma}^{-}h^2}{Mu_3+Ch^2} \nonumber
\end{align}
Due to (\ref{Adm}) and $r_{\sigma}(u_2,u_4)=O(h)$, we further get that
\begin{align}
\mid\mathcal{F}_{P_1,\sigma}-F_{P_1,\sigma}\mid&\leq O(h^2)+r_{\sigma}^{+}O(h)+r_{\sigma}^{-}O(h) \nonumber \\
&= O(h^2). \nonumber
\end{align}
Therefore, the truncation error of the monotone FVE scheme (\ref{eqF29}) remains $O(h^2)$, which is a necessary condition for a scheme to guarantee the second-order accuracy.

As analysed in \cite{XBlanc}, the resulted flux is not consistent on non-rectangular meshes, if employing a difference quotient to approximate the gradient directly. Hence, usually a combination of two one-sided fluxes is employed to obtain a second-order monotone FV scheme. Nevertheless, the iso-parametric bilinear finite element space has a first-order approximate capability to the gradient. Then, the associated discrete flux presented in this paper still keeps a second-order accuracy even if the combination is omitted.
\subsection{Monotonicity and nonlinear iteration}
Let $\triangle t$ be the time size, $U^n=(U^{n}_{P})_{P\in \Omega_{h}}\in \mathrm{R}^{\mathcal{N}}$ be the vector of discrete unknowns at the $n-$th time level. Denote the area of dual element $K_{P}^{\ast}$ by $S_{K_{P}^{\ast}}$. Let $L$ be the diagonal matrix of the areas of dual elements.

Taking the basis functions of $V_h$ as test functions in (\ref{eqF29}) and using an implicit discretization in time, we obtain the full discretization for system (\ref{eq:F1})-(\ref{eq:F2}):
\begin{equation}\label{eqggF29}
L\frac{U^{n+1}-U^{n}}{\triangle t}+A(U^{n+1})U^{n+1}=f^{n+1}+g^{n+1}.
\end{equation}
Here, $A(U^{n+1})$ is a matrix corresponding to the spacial discretization, and $f^{n+1}=(f_{P}^{n+1})_{P\in \Omega_h}$ and $g^{n+1}=(g_{P}^{n+1})_{P\in\Omega_h}$ are two vectors corresponding to the source term and boundary condition. Finally, the equation (\ref{eqggF29}) can be further rewritten as
\begin{equation}\label{eqF30}
(L+\triangle tA(U^{n+1}))U^{n+1}=\triangle t(f^{n+1}+g^{n+1})+LU^{n}.
\end{equation}

For any vector $U\in \mathrm{R}^{\mathcal{N}}$ and $U\geq 0$, the matrix $A(U)$ has the following properties:
\begin{enumerate}
  \item All diagonal entries of matrix $A(U)$ are positive;
  \item All off-diagonal entries of $A(U)$ are non-positive;
  \item Column sum corresponding to the interior node is 0 and column sum corresponding to the boundary node is positive.
\end{enumerate}

There exist many nonlinear iteration methods to solve system (\ref{eqF30}). Since the efficiency of nonlinear iteration methods is not the focused issue in this paper, we just employ a simple nonlinear iteration, the relaxed version of the Picard iteration method, to solve the nonlinear system (\ref{eqF30}). This relaxed version demonstrates more robust behavior \cite{Wang4} and the nonlinear iterations are all convergent in our numerical tests. We choose a small value $\varepsilon_{non}>0$ and initial vector $\widehat{U}^{0}\geq0$, and repeat for $k=1,2,\cdots$,
\begin{enumerate}
  \item Solve $(L+\triangle t A(\widehat{U}^{k-1}))\widetilde{U}^{k}=\triangle t(f^{n+1}+g^{n+1})+LU^{n}$;
  \item $\widehat{U}^{k}=\widehat{U}^{k-1}+\omega(\widetilde{U}^{k}-\widehat{U}^{k-1})$, where $0<\omega\leq 1$ is the damping factor;
  \item Stop if $\parallel \widehat{U}^{k}-\widehat{U}^{k-1} \parallel \leq \varepsilon_{non} \parallel \widehat{U}^{0} \parallel$;
  \item The value of $U^{n+1}$ is then given by the last $\widehat{U}^{k}$.
\end{enumerate}

According to the definition of M-matrix \cite{XavierBlanc:apositivescheme}, it is noted that $A(U)^{T}$ is an M-matrix for any vector $U\geq 0$ and hence $(L+\triangle t A(U))^{T}$ is also an M-matrix. It holds from the properties of M-matrix \cite{XavierBlanc:apositivescheme} that the matrix $(L+\triangle t A(U))^{-T}$ has non-negative elements. Under the assumptions that $(f^{n+1}+g^{n+1})\geq 0, U^{0}\geq0$ and the linear systems in the above Picard iterations are exactly solved, it is readily evident that all iterative solutions are non-negative vectors.
\subsection{Existence of a solution for the nonlinear FVE scheme}
It is easy to observe that the monotone FVE scheme (\ref{eqF29}) is regular with respect to $u_h$ for any fixed positive constant $M$. Accordingly, the existence of a solution for nonlinear system (\ref{eqF30}) could be obtained by the Brouwer's theorem.

\begin{theorem}
If $f^{n+1}+g^{n+1}\geq0$ and $U^{n}\geq0$, then system (\ref{eqF30}) has a solution $U^{n+1}$ for any fixed positive constant $M.$
\end{theorem}
\begin{proof}
Let $C_0$ be a non-negative constant. Define a compact set
\begin{equation*}
\mathcal{C}=\bigg\{ w=(w_P)_{P\in \Omega_h} \in \mathrm{R}^{\mathcal{N}}: w\geq0, \sum_{P\in \Omega_h} S_{K_{P}^{\ast}} w_P\leq C_0 \bigg\}
\end{equation*}
and a mapping $\phi : \mathcal{C} \mapsto \mathrm{R}^{\mathcal{N}}$ such that
\begin{equation}
\phi(w)=(L+\triangle tA(w))^{-1}(\triangle t(f^{n+1}+g^{n+1})+LU^{n}).\label{eqgF30}
\end{equation}
We need to prove that $\phi$ has a fixed point in order to prove that equation (\ref{eqF30}) has a solution.

In view of the discussion in the above subsection, we prove that the matrix $(L+\triangle tA(w))^{-1}$ has non-negative elements. Since $\triangle t(f^{n+1}+g^{n+1})+LU^{n}\geq 0$, we have
\begin{equation*}
\forall~w\in \mathcal{C}, \phi(w) \geq 0.
\end{equation*}
Rewrite (\ref{eqgF30}) as
\begin{equation}
(L+\triangle tA(w))\phi(w)=(\triangle t(f^{n+1}+g^{n+1})+LU^{n}).\label{epqgF30}
\end{equation}
Multiplying (\ref{epqgF30}) by the constant vector $(1,\ldots,1)$ on the left, we deduce that
\begin{equation}
\sum_{P\in \Omega_h} S_{K_{P}^{\ast}}[\phi(w)]_P\leq \sum_{P\in \Omega_h}(\triangle t(f_{P}^{n+1}+g_{P}^{n+1})+S_{K_{P}^{\ast}}U^{n}_P).\label{efpqgF30}
\end{equation}
where we have used the facts that $\phi(w) \geq 0$ and all column sums in $A(w)$ are no-negative.

The righthand side of (\ref{efpqgF30}) is a non-negative constant. We take $C_0=\sum_{P\in \Omega_h}(\triangle t(f_{P}^{n+1}+g_{P}^{n+1})+S_{K_{P}^{\ast}}U^{n}_P)$.
Then, $\phi$ maps $\mathcal{C}$ into itself.

The set $\mathcal{C}$ is a convex compact subset of $\mathrm{R}^{\mathcal{N}}$. Since each coefficient of $A(w)$ is a continuous function of $w$ and $A(w)\mapsto (L+\triangle tA(w))^{-1}$ is continuous from the set M-matrices to the set of matrices, $\phi$ is continuous. Hence, we may apply Brouwer's theorem, which implies that $\phi$ has a fixed point in $\mathcal{C}$, i.e., the system (\ref{eqF30}) has a solution. The theorem 3.1 is proved.
\end{proof}
\section{Numerical experiments}
In this section, we examine numerical performance of the monotone FVE method proposed in this paper. The monotonicity and convergence study of the discrete solution are presented for steady diffusion problems on both uniform and random distorted quadrilateral meshes. Moreover, to illustrate the efficiency of our scheme, we apply it to solve an equilibrium radiation diffusion equation. In all examples, we compare this new monotone FVE scheme with the standard FVE method\cite{ShuShiHuangYunqing,Refff16,Shu16}.

In the nonlinear iterations, the prescribed tolerance $\varepsilon_{non}$ is $10^{-7}$, and we use the GMRES method\cite{GMRES} to solve the linear systems. In our numerical experiments, we use the following discrete $L_{2}$-norm and discrete $H^1$-semi-norm to evaluate the approximation errors:
\begin{align*}
\parallel u-u_h\parallel_{0,h}&=\big(\sum_{K^{\ast}_{P}\in \mathcal{T}_{h}^{\ast}}  S_{K^{\ast}_{P}}\mid u(P)-u_h(P)\mid^2\big)^{1/2},\\
\mid u-u_h\mid_{1,h}&=\big(\sum_{K \in \mathcal{T}}S_{K}\mid \nabla u(Q)-\nabla u_h(Q)\mid^2\big)^{1/2},
\end{align*}
where $S_{K^{\ast}_{P}}$ and $S_{K}$ are the areas of $K^{\ast}_{P}$ and $K$, respectively. The rate of convergence is obtained by the following formula:
\begin{eqnarray*}
Rate=\frac{\log[E(h_2)/E(h_1)]}{\log(h_2/h_1)},
\end{eqnarray*}
where $h_1, h_2$ denote the mesh sizes of two successive meshes, and $E(h_1), E(h_2)$ are the corresponding errors.
\begin{figure}[H]
\centering
\subfigure[Uniform mesh]{
\label{Figa}
\includegraphics[width=0.4\textwidth]{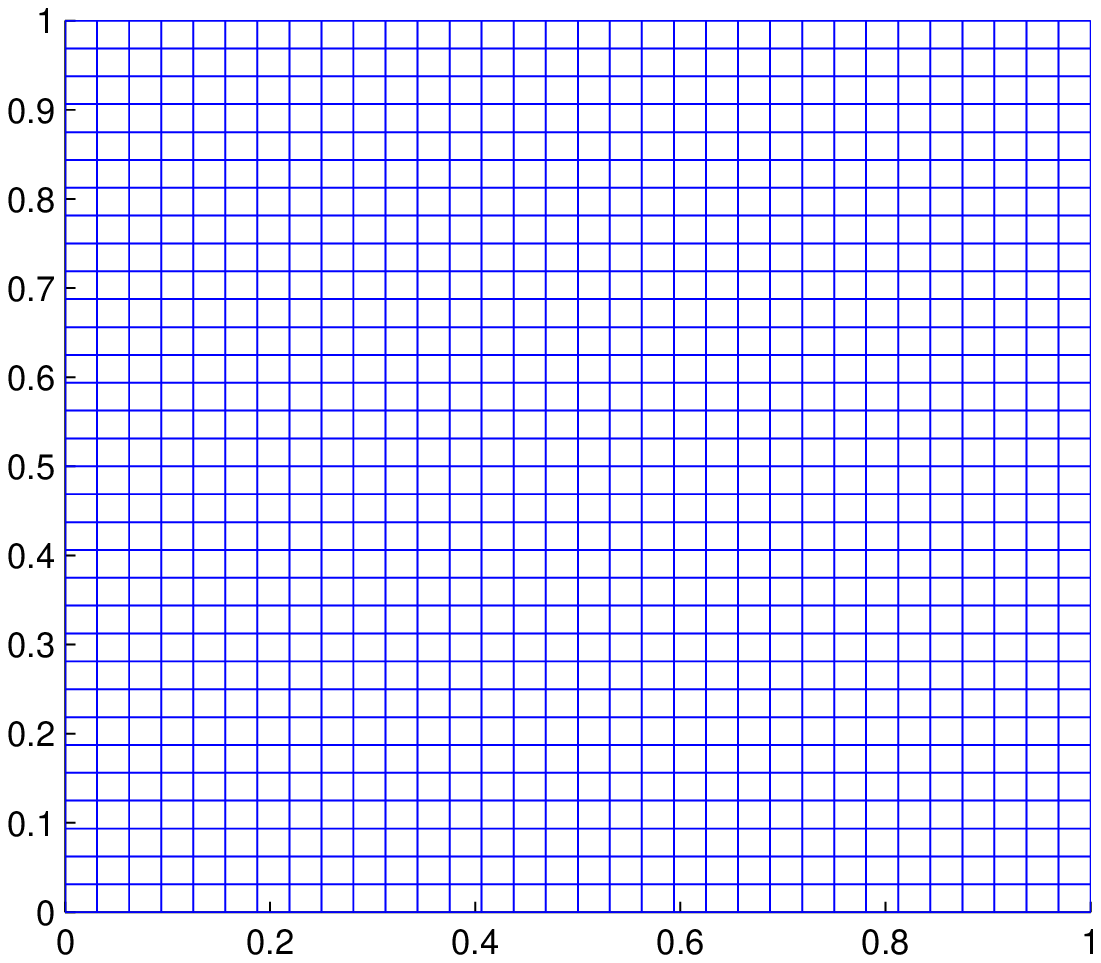}}
\subfigure[Random quadrilateral mesh]{
\label{Figb}
\includegraphics[width=0.4\textwidth]{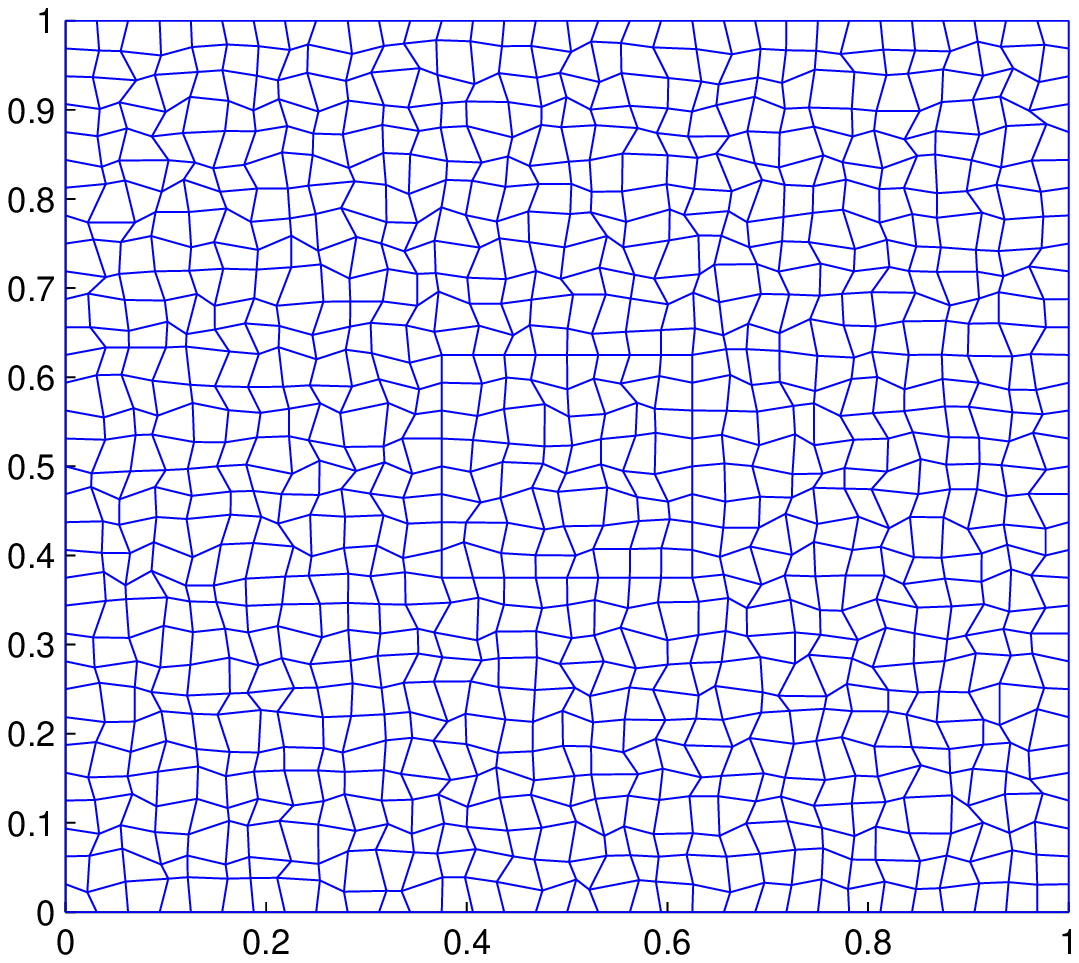}}
\caption{Two types of meshes for the monotonicity test.}
\label{Fig4}
\end{figure}

The numerical tests are performed on uniform meshes and random distorted quadrilateral meshes, as shown in Figure \ref{Fig4}. The distorted meshes are constructed from the uniform meshes by adding a random perturbation to the interior nodes. Unlike the convergence study for the standard FVE methods in \cite{Ref16}, the distortion in this paper is performed on each refinement level.

\subsection{Monotonicity}
The computational domain is a bi-unit square. The forcing function is taken as
\begin{equation*}
f=\left\{
    \begin{array}{ll}
      1, & \hbox{if}~~(x,y)\in [3/8,5/8]^2, \\
      0, & \hbox{otherwise.}
    \end{array}
  \right.
\end{equation*}
The diffusion tensor is given by
\begin{equation*}
\kappa=\left(
         \begin{array}{cc}
           y^2+\alpha x^2 &-(1-\alpha)xy  \\
           -(1-\alpha)xy & \alpha y^2+x^2\\
         \end{array}
       \right)
\end{equation*}
Consider a steady problem (\ref{eq:F1})-(\ref{eq:F2}) with boundary condition $10^{-9}(\kappa \nabla u)\cdot \mathbf{n}+u=0.$
In this paper, we take $\alpha=0.01$.

We test the proposed monotone FVE scheme on uniform meshes and random quadrilateral meshes (see Figure \ref{Fig4}). The exact solution $u(x,y)$ is unknown; nevertheless, the maximum principle indicates that it is non-negative. The numerical solutions obtained by the standard FVE scheme are shown in Figure \ref{Fig5}, and the counterparts for the monotone FVE scheme are shown in Figures \ref{Fig6}. In these figures, we observe that the standard FVE scheme produces negative values; however, the monotone FVE scheme preserves the positivity of the continuous solution.
\begin{figure}[H]
\centering
\subfigure[$u_{min}=-1.158409e-6,u_{max}=0.135485$]{
\label{Figggaa}
\includegraphics[width=0.49\textwidth]{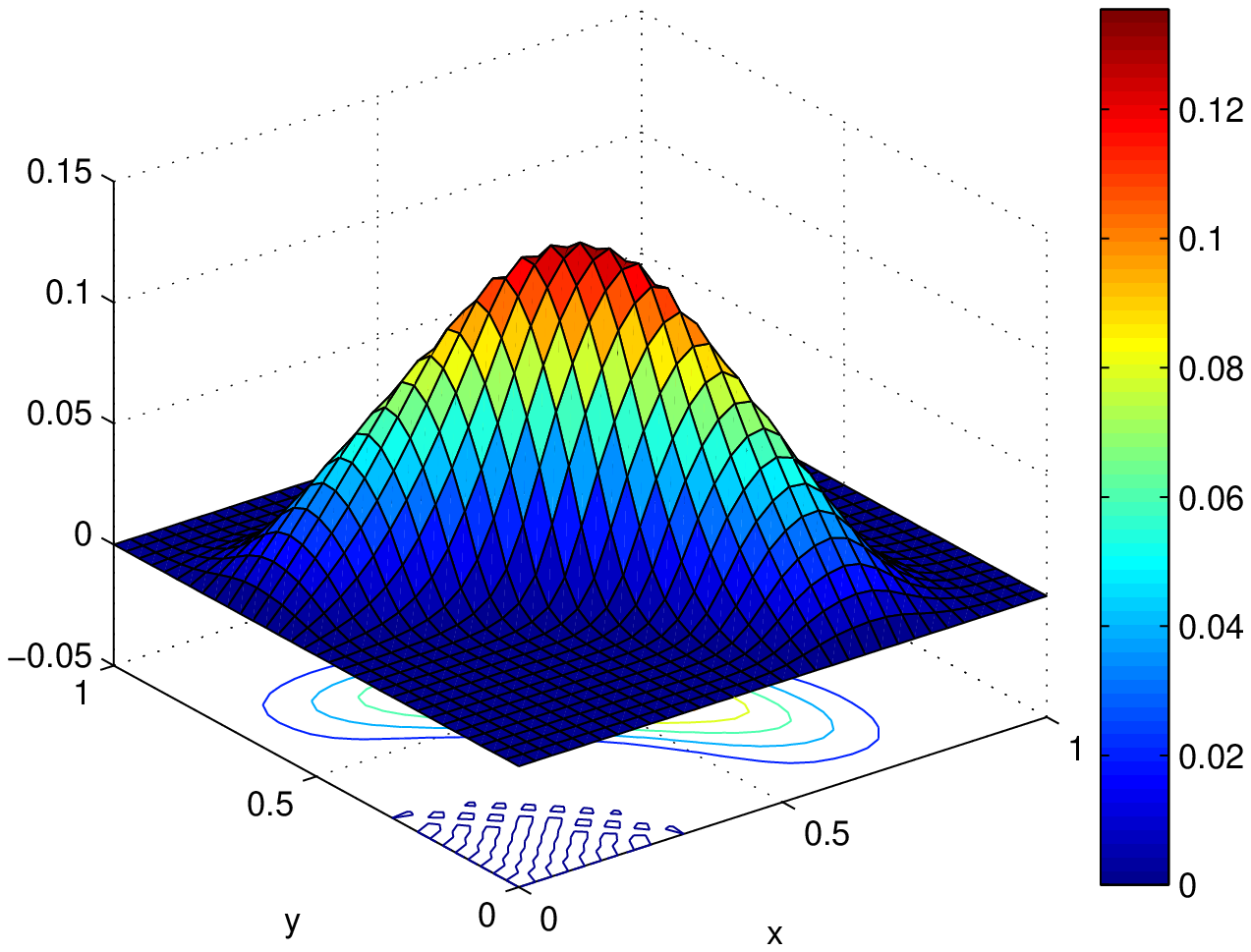}}
\subfigure[$u_{min}=-3.9916e-3,u_{max}=0.136371$]{
\label{Figggbb}
\includegraphics[width=0.49\textwidth]{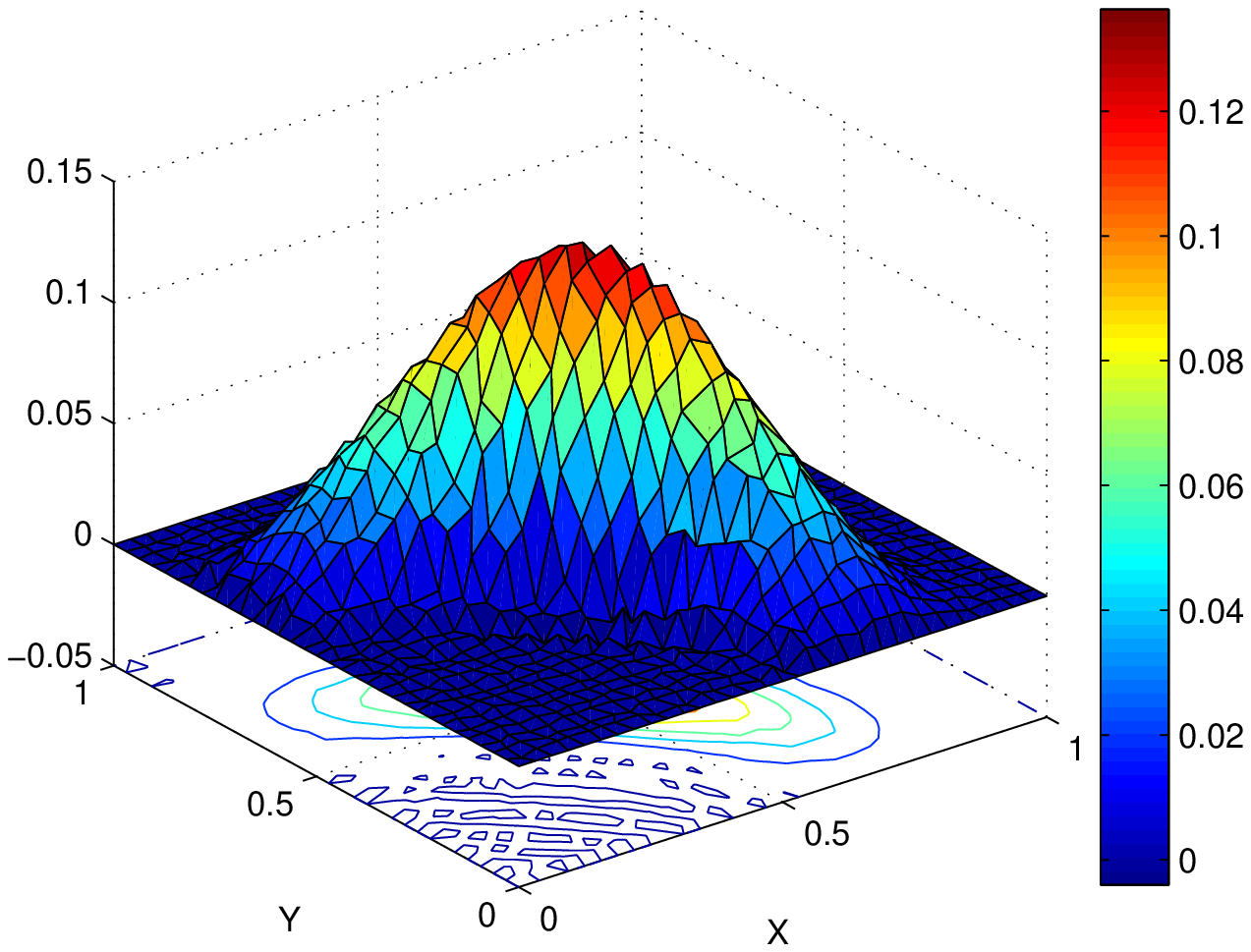}}
\caption{Solution profile on uniform (left) and random quadrilateral meshes (right) for standard FVE method.}
\label{Fig5}
\end{figure}

\begin{figure}[H]
\centering
\subfigure[$u_{min}=2.02680e-13,u_{max}=0.13586$]{
\label{Figgga}
\includegraphics[width=0.49\textwidth]{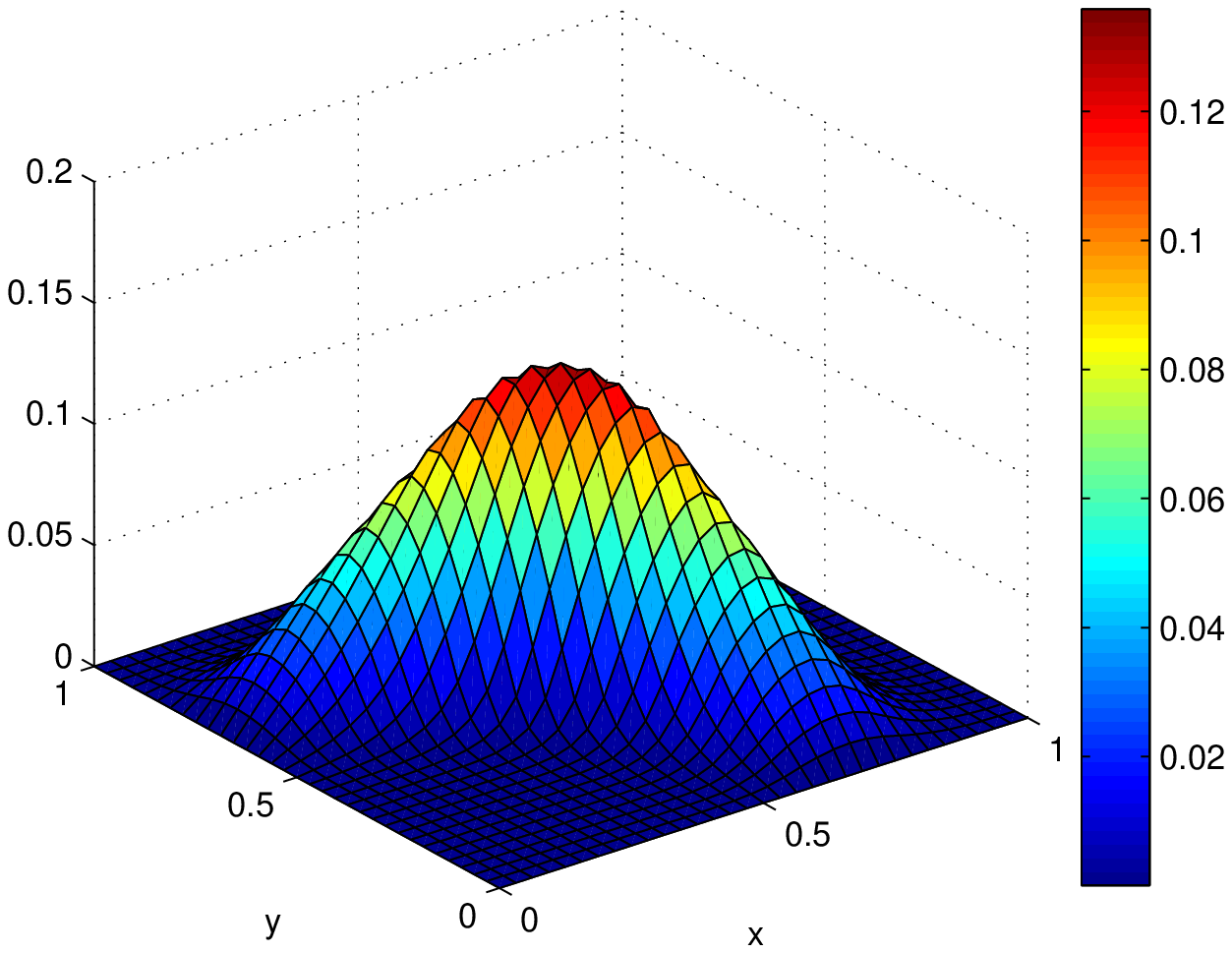}}
\subfigure[$u_{min}=2.02681e-13,u_{max}=0.13734$]{
\label{Figggb}
\includegraphics[width=0.49\textwidth]{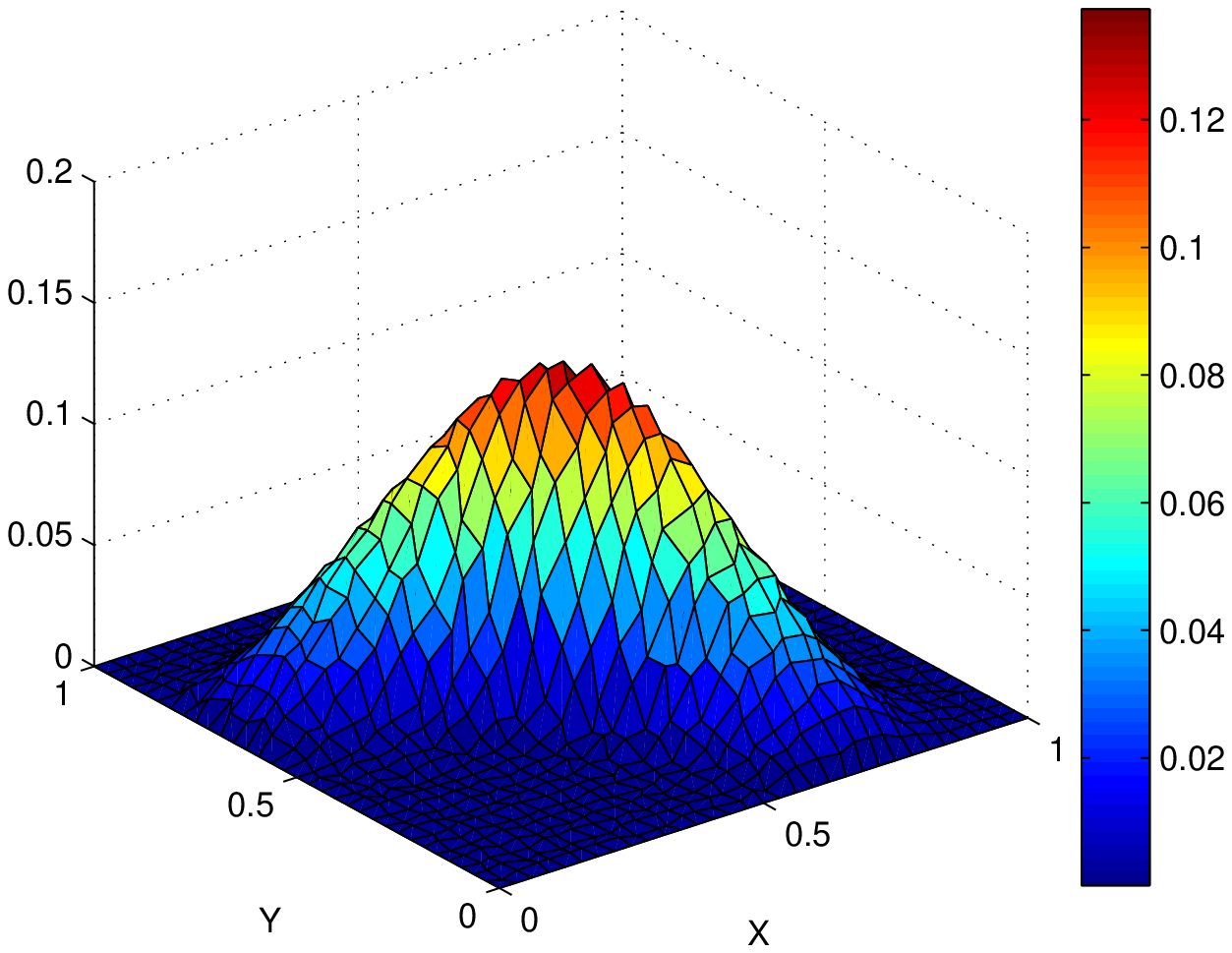}}
\caption{Solution profile on uniform (left) and random quadrilateral meshes (right) for monotone FVE sheme.}
\label{Fig6}
\end{figure}
\subsection{Convergence}
$\mathbf{Example~1}$: Convergence study for the isotropic diffusion problem. Let $\Omega=(0,1)^2$, the exact solution and diffusion coefficient be as follows:
\begin{equation*}
u(x,y)=\sin(\frac{\pi x}{2})\sin(\frac{\pi y}{2})+\frac{\pi}{2}, ~~~\kappa=\text{exp}(-x-y).
\end{equation*}

The source term $f$ and the boundary data $(\kappa \nabla u)\cdot \mathbf{n}+u=g$ are set accordingly to the exact solution. Since the positivity of the discrete unknowns is a requirement for the proposed monotone FVE scheme, the exact solution, and $f+g$ are non-negative in this example. Numerical errors and convergence rates for standard FVE method and monotone FVE scheme are presented in Table \ref{tab:1} and Table \ref{tab:2}, respectively.

\begin{table}
\caption{Standard FVE method: Errors and convergence rates for $\mathbf{Example 1}$.}
\label{tab:1}       
\begin{tabular}{lllllllll}
\hline\noalign{\smallskip}
     &\multicolumn{4}{c}{Uniform meshes}&\multicolumn{4}{c}{Random distorted meshes}\\
\cline{2-5}\cline{6-9}
Mesh level&$\parallel u-u_h\parallel_{0,h}$&Order&$\mid u-u_h\mid_{1,h}$&Order&$\parallel u-u_h\parallel_{0,h}$&Order&$\mid u-u_h\mid_{1,h}$&Order\\
\noalign{\smallskip}\hline\noalign{\smallskip}
1&$4.3080\times10^{-3}$&$-$&$1.2970\times10^{-2}$&$-$&$4.1090\times10^{-3}$&$-$&$1.7767\times10^{-2}$&$-$   \\
2&$1.0642\times10^{-3}$&2.0172&$4.0432\times10^{-3}$&1.6816&$1.3476\times10^{-3}$&1.6985&$8.0991\times10^{-3}$&1.1968   \\
3&$2.6630\times10^{-4}$&1.9986&$1.2056\times10^{-3}$&1.7457&$3.1712\times10^{-4}$&2.1088&$3.7088\times10^{-3}$&1.1384   \\
4&$6.6784\times10^{-5}$&1.9954&$3.4669\times10^{-4}$&1.7980&$1.0000\times10^{-4}$&1.6895&$1.8126\times10^{-3}$&1.0480   \\
5&$1.6738\times10^{-5}$&1.9963&$9.7073\times10^{-5}$&1.8365&$2.7870\times10^{-5}$&1.8788&$9.2189\times10^{-4}$&0.9942   \\
\noalign{\smallskip}\hline
\end{tabular}
\end{table}
\begin{table}
\caption{Monotone FVE scheme: Errors and convergence rates for $\mathbf{Example 1}$.}
\label{tab:2}       
\begin{tabular}{lllllllll}
\hline\noalign{\smallskip}
     &\multicolumn{4}{c}{Uniform meshes}&\multicolumn{4}{c}{Random distorted meshes}\\
\cline{2-5}\cline{6-9}
 Mesh level&$\parallel u-u_h\parallel_{0,h}$&Order&$\mid u-u_h\mid_{1,h}$&Order&$\parallel u-u_h\parallel_{0,h}$&Order&$\mid u-u_h\mid_{1,h}$&Order\\
\noalign{\smallskip}\hline\noalign{\smallskip}
1&$3.1348\times10^{-3}$&$-$&$1.4965\times10^{-2}$&$-$&$4.8452\times10^{-3} $&$-$&$2.3929\times10^{-2}$&$-$   \\
2&$8.0581\times10^{-4}$&1.9598&$4.7442\times10^{-3}$&1.6573&$1.5094\times10^{-3}$&1.7768&$1.0179\times10^{-2}$&1.3022\\
3&$2.0610\times10^{-4}$&1.9670&$1.4366\times10^{-3}$&1.7235&$3.7475\times10^{-4}$&2.0307&$4.9751\times10^{-3}$&1.0434\\
4&$5.1667\times10^{-5}$&1.9960&$4.1797\times10^{-4}$&1.7811&$1.1363\times10^{-4}$&1.7469&$2.4630\times10^{-3}$&1.0292\\
5&$1.2361\times10^{-5}$&2.0634&$1.1804\times10^{-4}$&1.8241&$3.0842\times10^{-5}$&1.9177&$1.2319\times10^{-3}$&1.0188\\
\noalign{\smallskip}\hline
\end{tabular}
\end{table}
$\mathbf{Example~2}$: Convergence study for the anisotropic diffusion problem. Let $\Omega=(0,1)^2$, the exact solution and diffusion coefficient be as follows:
\begin{equation*}
u(x,y)=\sin(\pi x)\sin(\pi y)+0.0041\pi,
\end{equation*}
\begin{equation*}
\kappa=\left(
         \begin{array}{cc}
          \cos\theta & \sin\theta \\
           -\sin\theta & \cos\theta \\
         \end{array}
       \right)\left(
                \begin{array}{cc}
                  k_1 & 0 \\
                  0 & k_2 \\
                \end{array}
              \right)\left(
                       \begin{array}{cc}
           \cos\theta & -\sin\theta \\
           \sin\theta & \cos\theta \\
                        \end{array}
                     \right)
,
\end{equation*}
where $k_1=0.1, k_2=4, \theta=\pi/6$. The source term $f$ and the boundary data $g=(\kappa \nabla u)\cdot \mathbf{n}+1000u$ are set accordingly to the exact solution. Likewise, the exact solution, and $f+g$ are non-negative in this example. Numerical errors and convergence rates for standard FVE method and monotone FVE scheme are presented in Table \ref{tab:3} and Table \ref{tab:4}, respectively.
\begin{table}
\caption{Standard FVE Method: Errors and convergence rates for $\mathbf{Example 2}$.}
\label{tab:3}       
\begin{tabular}{lllllllll}
\hline\noalign{\smallskip}
     &\multicolumn{4}{c}{Uniform meshes}&\multicolumn{4}{c}{Random distorted meshes}\\
\cline{2-5}\cline{6-9}
Mesh level&$\parallel u-u_h\parallel_{0,h}$&Order&$\mid u-u_h\mid_{1,h}$&Order&$\parallel u-u_h\parallel_{0,h}$&Order&$\mid u-u_h\mid_{1,h}$&Order\\
\noalign{\smallskip}\hline\noalign{\smallskip}
1&$5.6362\times10^{-3}$&$-$&$1.8126\times10^{-2}$&$-$&$6.8585\times10^{-3}$&$-$&$7.2047\times10^{-2}$&$-$  \\
2&$1.3960\times10^{-3}$&2.0134&$4.5117\times10^{-3}$&2.0063&$1.8195\times10^{-3}$&2.0216&$3.3136\times10^{-2}$&1.1833\\
3&$3.4820\times10^{-4}$&2.0033&$1.1266\times10^{-3}$&2.0016&$4.5899\times10^{-4}$&2.0075&$1.5466\times10^{-2}$&1.1106\\
4&$8.7001\times10^{-5}$&2.0008&$2.8159\times10^{-4}$&2.0003&$1.1934\times10^{-4}$&1.9719&$7.7407\times10^{-3}$&1.0132\\
5&$2.1747\times10^{-5}$&2.0002&$7.0393\times10^{-5}$&2.0009&$3.0536\times10^{-5}$&2.0045&$3.8458\times10^{-3}$&1.0287\\
\noalign{\smallskip}\hline
\end{tabular}
\end{table}
\begin{table}
\caption{Monotone FVE scheme: Errors and convergence rates for $\mathbf{Example 2}$.}
\label{tab:4}       
\begin{tabular}{lllllllll}
\hline\noalign{\smallskip}
     &\multicolumn{4}{c}{Uniform meshes}&\multicolumn{4}{c}{Random distorted meshes}\\
\cline{2-5}\cline{6-9}
Mesh level&$\parallel u-u_h\parallel_{0,h}$&Order&$\mid u-u_h\mid_{1,h}$&Order&$\parallel u-u_h\parallel_{0,h}$&Order&$\mid u-u_h\mid_{1,h}$&Order\\
\noalign{\smallskip}\hline\noalign{\smallskip}
1&$1.8732\times10^{-2}$&$-$&$2.1715\times10^{-2}$&$-$&$1.0505\times10^{-2}$&$-$&$1.3502\times10^{-1}$&$-$  \\
2&$4.6227\times10^{-3}$&2.0186&$5.4135\times10^{-3}$&2.0040&$2.8031\times10^{-3}$&2.0127&$6.5878\times10^{-2}$&1.0933\\
3&$1.1518\times10^{-3}$&2.0048&$1.3700\times10^{-3}$&1.9823&$8.2891\times10^{-4}$&1.7759&$3.7008\times10^{-2}$&0.8405\\
4&$2.8767\times10^{-4}$&2.0014&$3.6002\times10^{-4}$&1.9280&$2.4248\times10^{-4}$&1.7994&$1.9862\times10^{-2}$&0.9110\\
5&$7.1895\times10^{-5}$&2.0004&$1.0489\times10^{-4}$&1.7791&$7.1023\times10^{-5}$&1.8057&$1.0479\times10^{-2}$&0.9403\\
\noalign{\smallskip}\hline
\end{tabular}
\end{table}

$\mathbf{Example~3}$: Convergence study for the diffusion problem with a discontinuous coefficient. Let $\Omega=(0,1)^2$ and the exact solution and diffusion coefficient be as follows:
\begin{equation*}
u(x,y)=\left\{
         \begin{array}{ll}
           \sin(\frac{\pi x}{2})\sin(\pi y)+2\sqrt{3}\pi \times 10^{-4}, & \hbox{if}~~ x\leq 1/3, \\
           \frac{\sqrt{3}}{3}\sin(\pi x)\sin(\pi y)+2\sqrt{3}\pi \times 10^{-4},& \hbox{if}~~ x> 1/3.
         \end{array}
       \right.
\end{equation*}
\begin{equation*}
\kappa=\left(
         \begin{array}{cc}
          4 & 0 \\
          0 & 1 \\
         \end{array}
       \right)
,~\hbox{if}~~ x\leq 1/3;~~ \kappa=\left(
         \begin{array}{cc}
          6 & 0 \\
          0 & 1 \\
         \end{array}
       \right)
,~\hbox{if}~~ x> 1/3.
\end{equation*}
The source term $f$ and the boundary data $g=(\kappa \nabla u)\cdot \mathbf{n}+10000u$ are set accordingly to the exact solution. We have examined that the exact solution, and $f+g$ are non-negative. Numerical errors and convergence rates for standard FVE method and monotone FVE scheme are presented in Table \ref{tab:5} and Table \ref{tab:6}, respectively.
\begin{table}
\caption{Standard FVE method: Errors and convergence rates for $\mathbf{Example 3}$.}
\label{tab:5}       
\begin{tabular}{lllllllll}
\hline\noalign{\smallskip}
     &\multicolumn{4}{c}{Uniform meshes}&\multicolumn{4}{c}{Random distorted meshes}\\
\cline{2-5}\cline{6-9}
Mesh level&$\parallel u-u_h\parallel_{0,h}$&Order&$\mid u-u_h\mid_{1,h}$&Order&$\parallel u-u_h\parallel_{0,h}$&Order&$\mid u-u_h\mid_{1,h}$&Order\\
\noalign{\smallskip}\hline\noalign{\smallskip}
1&$3.1045\times10^{-3}$&$-$&$4.0033\times10^{-3}$&$-$&$3.9770\times10^{-3}$&$-$&$2.5651\times10^{-2}$&$-$\\
2&$7.7297\times10^{-4}$&2.0058&$9.9901\times10^{-4}$&2.0026&$1.2873\times10^{-3}$&1.6441&$1.6449\times10^{-2}$&0.6293\\
3&$1.9304\times10^{-4}$&2.0015&$2.4971\times10^{-4}$&2.0002&$3.3317\times10^{-4}$&1.9787&$8.1944\times10^{-3}$&1.0200\\
4&$4.8248\times10^{-5}$&2.0003&$6.2500\times10^{-5}$&1.9983&$8.8884\times10^{-5}$&1.9431&$4.0906\times10^{-3}$&1.0217\\
5&$1.2063\times10^{-5}$&1.9998&$1.5841\times10^{-5}$&1.9801&$2.2002\times10^{-5}$&2.0020&$2.1065\times10^{-3}$&0.9759\\
\noalign{\smallskip}\hline
\end{tabular}
\end{table}
\begin{table}
\caption{Monotone FVE scheme: Errors and convergence rates for $\mathbf{Example 3}$.}
\label{tab:6}       
\begin{tabular}{lllllllll}
\hline\noalign{\smallskip}
     &\multicolumn{4}{c}{Uniform meshes}&\multicolumn{4}{c}{Random distorted meshes}\\
\cline{2-5}\cline{6-9}
Mesh level&$\parallel u-u_h\parallel_{0,h}$&Order&$\mid u-u_h\mid_{1,h}$&Order&$\parallel u-u_h\parallel_{0,h}$&Order&$\mid u-u_h\mid_{1,h}$&Order\\
\noalign{\smallskip}\hline\noalign{\smallskip}
1&$2.7071\times10^{-3}$&$-$&$3.0775\times10^{-3}$&$-$&$4.2274\times10^{-3}$&$-$&$3.2208\times10^{-2}$&$-$\\
2&$6.7108\times10^{-4}$&2.0121&$1.2207\times10^{-3}$&1.3340&$1.4420\times10^{-3}$&1.6386&$2.1269\times10^{-2}$&0.6322\\
3&$1.6554\times10^{-4}$&2.0193&$3.8384\times10^{-4}$&1.6691&$3.9136\times10^{-4}$&1.9009&$1.0640\times10^{-2}$&1.0095\\
4&$4.2234\times10^{-5}$&1.9707&$1.1397\times10^{-4}$&1.7518&$9.8116\times10^{-5}$&2.0253&$5.1350\times10^{-3}$&1.0665\\
5&$1.0510\times10^{-5}$&2.0066&$2.9818\times10^{-5}$&1.9343&$2.5751\times10^{-5}$&1.9672&$2.6463\times10^{-3}$&0.9749\\
\noalign{\smallskip}\hline
\end{tabular}
\end{table}
From the numerical results presented in Tables \ref{tab:1}-\ref{tab:6}, it is evident that the proposed monotone FVE scheme can achieve nearly the same convergence rates as the standard FVE method in both the discrete $L_{2}$-norm and discrete $H^1$-semi-norm, and the numerical errors are comparable for these two methods. On both uniform meshes and random distorted quadrilateral meshes, the proposed method have an approximate second-order convergence rate in the discrete $L_{2}$-norm and a higher than first-order convergence rate in the discrete $H^1$-semi-norm. Moreover, super-convergence in the discrete $H^1$-semi-norm can be still observed on uniform meshes.
\subsection{Equilibrium radiation diffusion problem}
\begin{figure}[H]
\centering
\subfigure[]{
\label{Figa}
\includegraphics[width=0.49\textwidth]{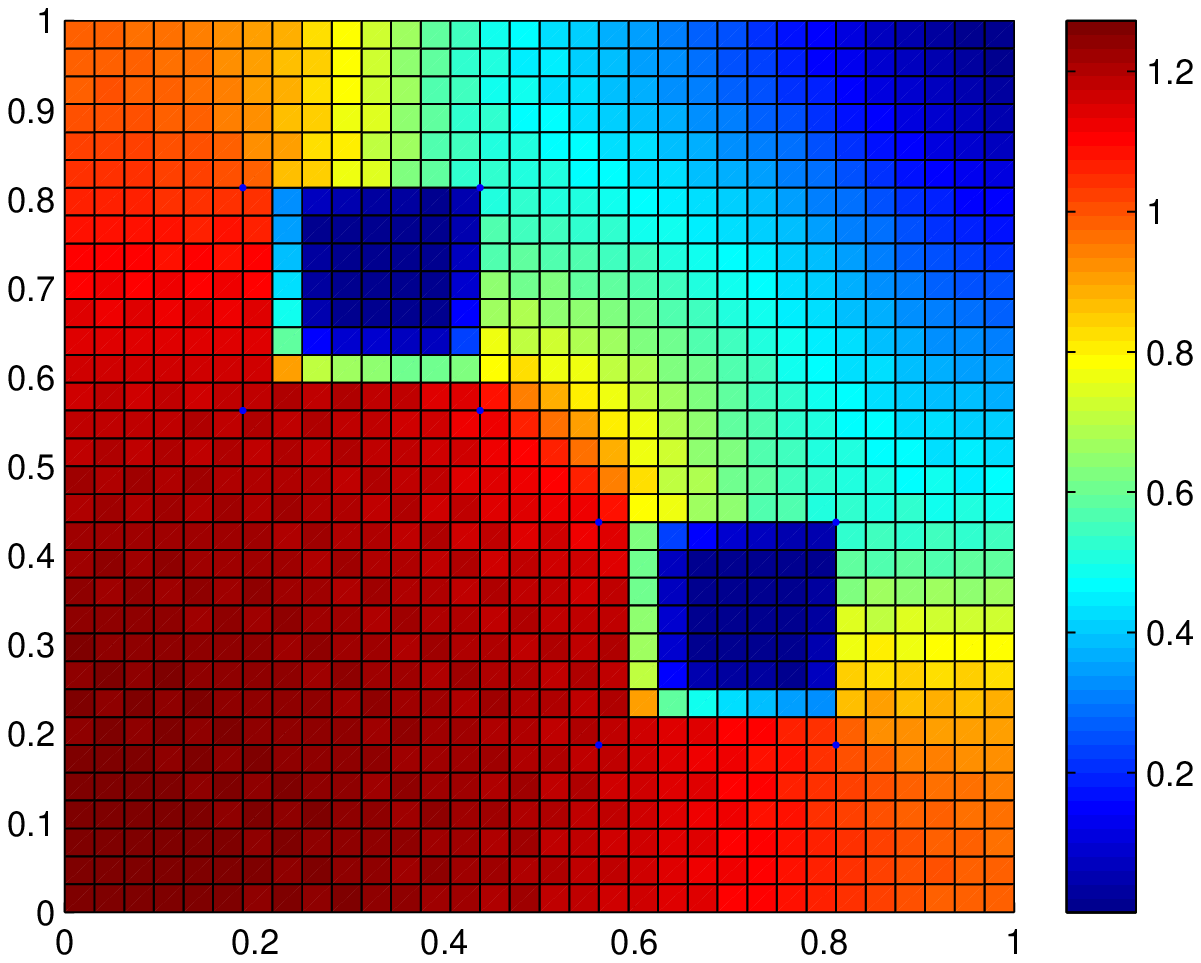}}
\subfigure[]{
\label{Figb}
\includegraphics[width=0.49\textwidth]{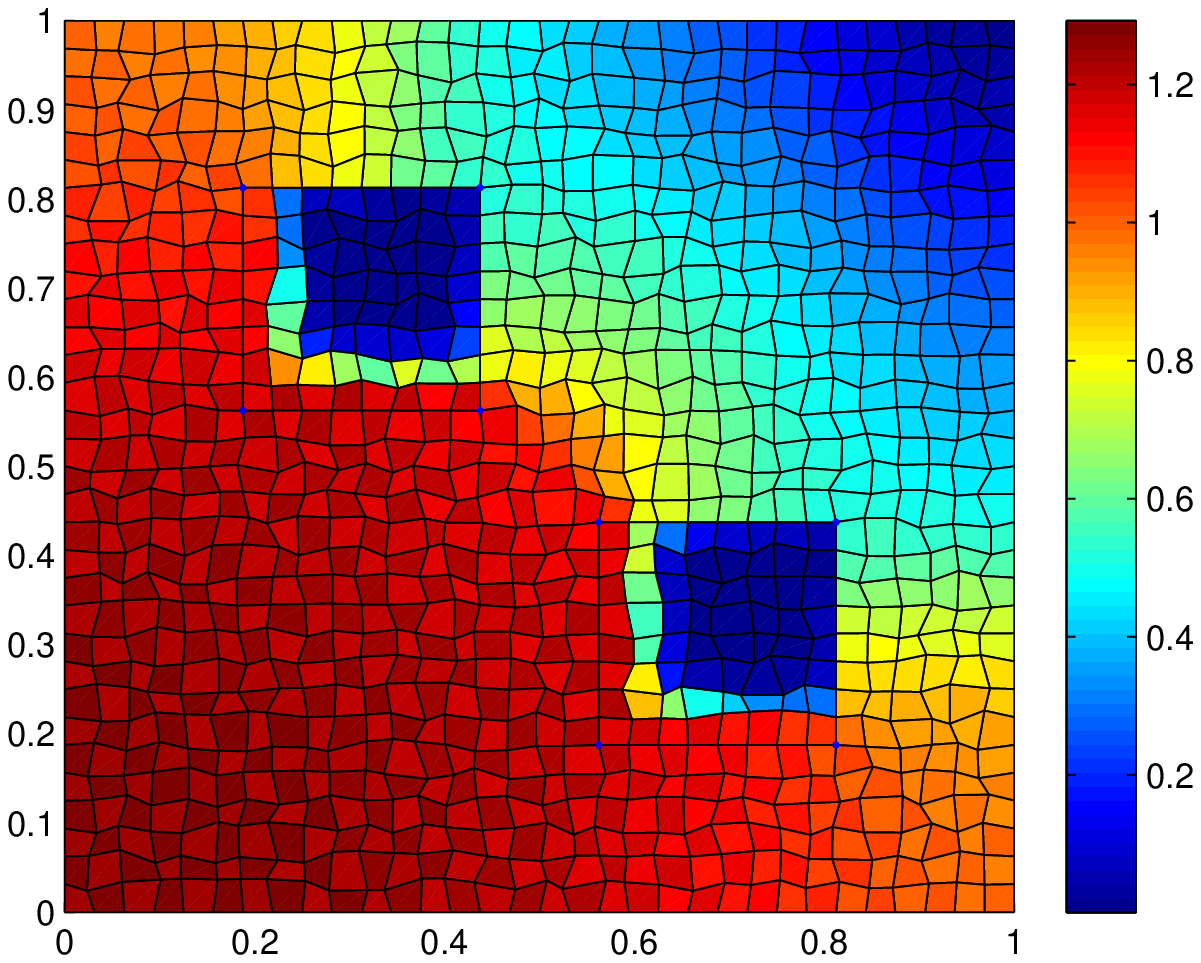}}
\caption{Solution profile on uniform(left) and random quadrilateral meshes(right) for standard FVE method.}
\label{Fig8}
\end{figure}

\begin{figure}[H]
\centering
\subfigure[]{
\label{Figa}
\includegraphics[width=0.49\textwidth]{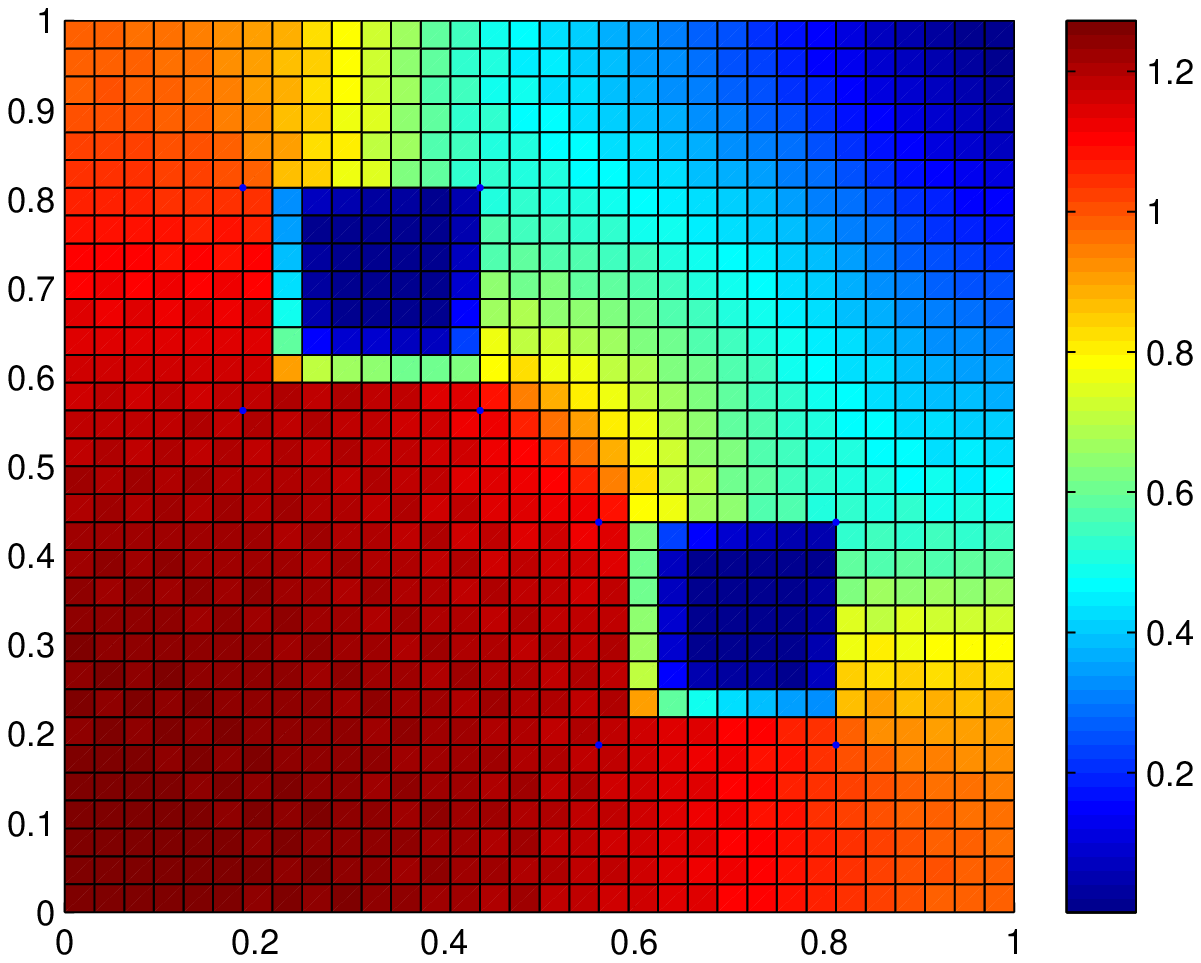}}
\subfigure[]{
\label{Figb}
\includegraphics[width=0.49\textwidth]{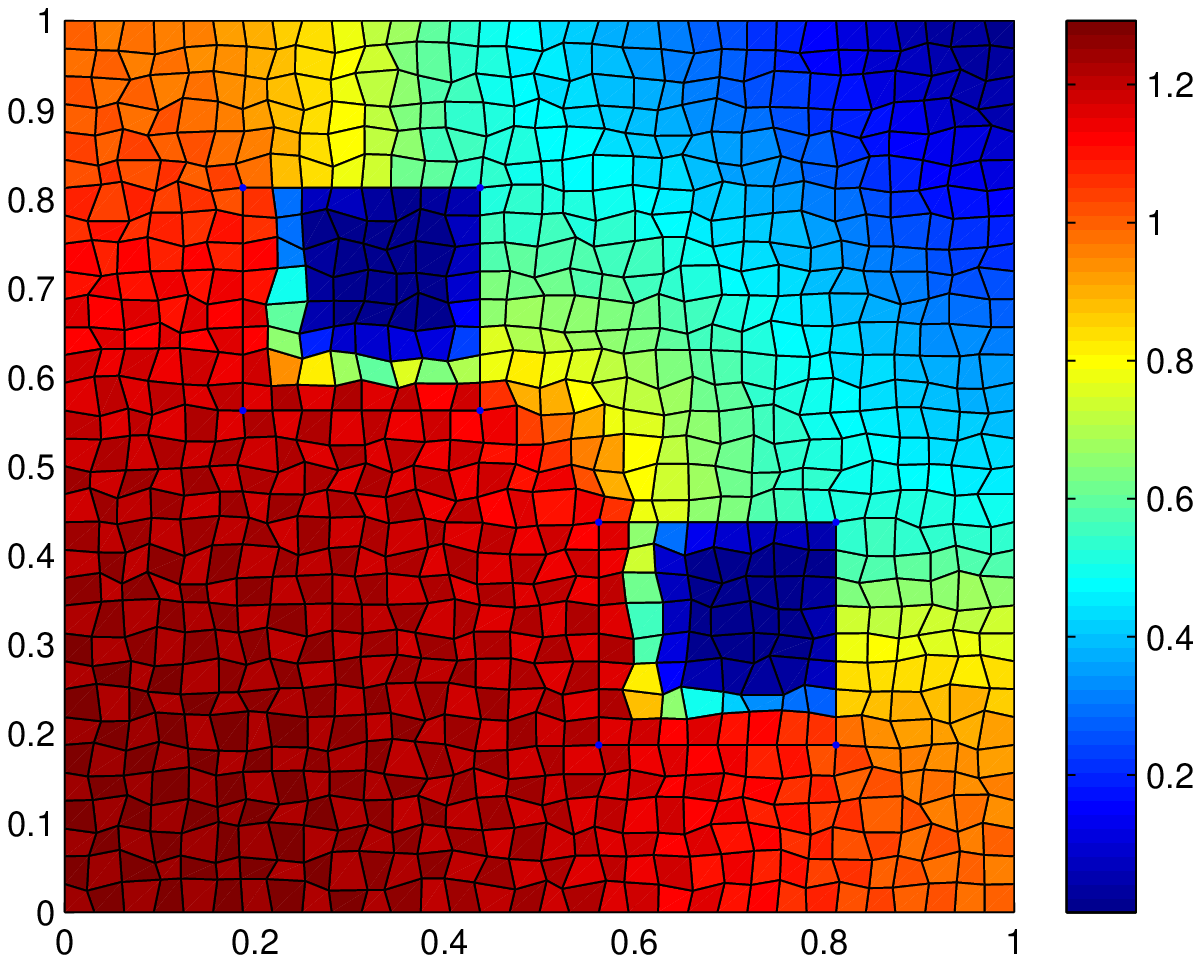}}
\caption{Solution profile on uniform(left) and random quadrilateral meshes(right) for monotone FVE scheme.}
\label{Fig9}
\end{figure}
To examine the computational efficiency of the new monotone FVE scheme, we apply it to the equilibrium radiation diffusion problem.
We provide a brief description of this equation. See \cite{Equilibrium1,Equilibrium2} for a more detailed discussion.
Consider the following nonlinear parabolic partial differential equation:
\begin{align*}
 \frac{\partial E}{\partial t}=\nabla \cdot D_{L}(E) \nabla E ~~~~~~ x\in \Omega, 0<t\leq T,
\end{align*}
where $\Omega=(0,1)^2$, $E$ is the dimensionless gray radiation energy density, and $D_L(E)$ is the Larsen¡¯s form of the flux-limited diffusion coefficient defined as
\begin{align*}
D_{L}(E)=\bigg( \frac{1}{D(E)^2}+\frac{\mid \nabla E\mid^2}{E^2}\bigg)^{-\frac{1}{2}}.
\end{align*}
Here, $D(E)=Z^{-3}E^{\frac{3}{4}}$, and $Z$ is the atomic number of the medium, which can be a discontinuous function. The value for $Z$ is 1 everywhere, except in the two obstacles defined by
\begin{align*}
\frac{3}{16}<x<\frac{7}{16},~~~~~\frac{9}{16}<y<\frac{13}{16}
\end{align*}
and
\begin{align*}
\frac{9}{16}<x<\frac{13}{16},~~~~~\frac{3}{16}<y<\frac{7}{16},
\end{align*}
where the value for $Z$ is 10. The boundary condition is
\begin{align*}
\frac{\partial E}{\partial x}\mid_{x=0}=\frac{\partial E}{\partial x}\mid_{x=1}=\frac{\partial E}{\partial y}\mid_{y=0}=\frac{\partial E}{\partial y}\mid_{y=1}=0.
\end{align*}
The initial radiation energy is given by
\begin{align*}
E(r)=0.001+100\text{exp}\big[-(\frac{r}{0.1})^{2}\big],
\end{align*}
where $r=\sqrt{x^2+y^2}.$

\begin{table}
\caption{$L_2$-norm of the solution for the proposed scheme.}
\label{tab:10}
\begin{center}
\begin{tabular}{lll}
\hline\noalign{\smallskip}
   & $L_2$-norm on uniform meshes & $L_2$-norm on random quadrilateral  meshes \\
\noalign{\smallskip}\hline\noalign{\smallskip}
Standard FVE method  & 0.8854 & 0.8868       \\
Monotone FVE scheme  & 0.8854 & 0.8867  \\
\noalign{\smallskip}\hline
\end{tabular}
\end{center}
\end{table}
\begin{table}
\caption{Average number of iterations and CPU time on random quadrilateral meshes.}
\label{tab:11}
\begin{center}
\begin{tabular}{llll}
\hline\noalign{\smallskip}
   & Number of nonlinear iterations & Number of linear iterations & CPU time\\
\noalign{\smallskip}\hline\noalign{\smallskip}
Standard FVE method & 20.64075 & 32.654  &   3.85(h)\\
Monotone FVE scheme & 20.67585 & 20.000  &   1.12(h) \\
\noalign{\smallskip}\hline
\end{tabular}
\end{center}
\end{table}
For the time discretization, we use the backward Euler method. We take the time step size as $5.0e-4$ and let the final state be $T=1.0$.

Solution profiles for standard FVE method and monotone FVE scheme are presented in Figures \ref{Fig8} and \ref{Fig9}, respectively. In comparing Figures \ref{Fig8} and \ref{Fig9}, we observe that the solution obtained by these two methods are nearly the same. Moreover, the variation tendencies of the solutions on random quadrilateral meshes are in accordance with those on rectangular meshes for monotone FVE scheme. In Table \ref{tab:10}, we give the $L_2$-norm of the solution on rectangular and random quadrilateral meshes. From this table, it is apparent that the $L_2$-norm of the solution for each scheme on random quadrilateral meshes is close to that on rectangular meshes.

Next, we give the comparison of computational costs between our scheme and the standard FVE scheme. Table \ref{tab:11} provides the average number of nonlinear iterations per time step, the average number of linear iterations per nonlinear iteration and total CPU time on random quadrilateral meshes. There are nine nonzero elements in each row for the linear system of the standard FVE method; however, there are only five nonzero elements in each row for the linear system of our scheme. Hence, our scheme reduces the costs by almost half compared with the standard FVE method when the total number of linear iterations is the same. Table \ref{tab:11} shows that the proposed monotone FVE scheme requires fewer linear iterations in each time step than the standard FVE method and the total CPU time is reduced significantly. Hence, the proposed monotone FVE scheme is more efficient than the standard FVE method.
\section{Conclusion}
In this paper, we construct a nonlinear positivity-preserving FVE scheme on quadrilateral meshes, where the one-sided flux is approximated by the standard iso-parametric bilinear element. The new scheme is positivity-preserving by applying a nonlinear two-point flux technique and is applicable for both continuous and discontinuous anisotropic problems. Compared with the FV schemes, the main feature of this monotone FVE scheme is that the decomposition of co-normals and the introduction of auxiliary unknowns are both avoided. Numerical results show that the proposed monotone FVE scheme do not produce non-physical solutions and can reach nearly optimal convergence rates in both the $L_2$-norm and $H^1$-semi-norm. From the performance in solving the equilibrium radiation diffusion equation, it can be observed that the presented scheme can significantly reduce computational costs compared with the standard FVE method.

\end{document}